\documentclass[12pt]{amsart}

\usepackage{amsfonts}
\usepackage{amssymb}
\usepackage{amsthm}
\usepackage{amsmath}
\usepackage{enumitem}
\usepackage{graphicx}

\usepackage{tikz}
\usetikzlibrary{decorations.pathreplacing}

\usepackage{hyperref}

\newtheorem{theorem}{Theorem}
\newtheorem{lemma}[theorem]{Lemma}
\newtheorem{remark}[theorem]{Remark}

\newtheorem{proposition}[theorem]{Proposition}

\newtheorem{corollary}[theorem]{Corollary}
\newtheorem{definition}[theorem]{Definition}
\numberwithin{equation}{section}

\newcommand{\R}{{\mathbb{R}}}
\newcommand{\bS}{{\mathbb{S}}}

\newcommand{\eps}{\varepsilon}

\DeclareMathOperator{\area}{area}
\DeclareMathOperator{\length}{length}
\DeclareMathOperator{\Div}{div}
\DeclareMathOperator{\supp}{supp}
\DeclareMathOperator{\Ric}{Ric}

\begin{document}
\title[Asymptotically flat 3-manifolds contain minimal planes]{Asymptotically flat three-manifolds contain minimal planes}

\author{Otis Chodosh}\address{Department of Mathematics\\Princeton
University\\Princeton, NJ 08544}
\address{School of Mathematics\\Institute for Advanced Study\\Princeton, NJ 08540}
\email{ochodosh@math.princeton.edu}
\author{Daniel Ketover}\address{Department of Mathematics\\Princeton
University\\Princeton, NJ 08544}
 \email{dketover@math.princeton.edu}

\begin{abstract}
Let $(M,g)$ be an asymptotically flat $3$-manifold containing no closed embedded minimal surfaces.  We prove that for every point $p\in M$ there exists a complete properly embedded minimal plane in $M$ containing $p$.  
\end{abstract}

\maketitle

\section{Introduction}
Given a point $p$ in $\mathbb{R}^3$ there are infinitely many minimal planes passing through $p$. 
However, for a general complete metric on $\R^{3}$ with infinite volume, it is not known if any unbounded minimal planes (or surfaces of any topology) exist. This is the topic of our main result:\footnote{{Added in proof:} Mazet--Rosenberg \cite{MazetRosenberg} have recently generalized Theorem \ref{main} to show that under the same hypothesis, there exists a minimal plane through any \emph{three} points.}
\begin{theorem}\label{main}
Let $(M,g)$ be an asymptotically flat $3$-manifold containing no closed embedded minimal surfaces.  For every point $p\in M$ there exists a complete properly embedded minimal plane in $M$ containing $p$.  \end{theorem}

The following notion of asymptotic flatness suffices in Theorem \ref{main}: $M$ is diffeomorphic\footnote{Note that the work of Meeks--Simon--Yau \cite{MeeksSimonYau} shows that a general asymptotically flat $3$-manifold with no compact minimal surfaces is automatically diffeomorphic to $\R^{3}$ (cf.\ \cite[Lemma 4.1]{HuiskenIlmanen}). } to $\R^{3}$ and in the associated coordinates, the metric satisfies $g = \bar g + b$, where $|b| + |x| |\bar D b| + |x|^{2}|\bar D^{2}b| = o(1)$ as $|x| \to\infty$ (where $\bar g$ is the Euclidean metric and $\bar D$ the Euclidean connection). We emphasize that no curvature assumption (e.g., non-negative scalar curvature) is included in the statement of Theorem \ref{main}. 

Our motivation for Theorem \ref{main} comes from Schoen--Yau's proof of the Positive Mass Theorem \cite{SchoenYau:PMT}.  A key aspect of their proof is showing that \emph{certain} stable minimal surfaces cannot exist in an asymptotically flat $3$-manifold with positive scalar curvature. This non-existence result has been refined in the works \cite{EichmairMetzger,Carlotto,ChodoshEichmairCarlotto} (cf.\ \cite{CarlottoSchoen}) so as to apply to \emph{any} unbounded embedded stable minimal surface. In particular, these works show that such surfaces cannot exist in asymptotically flat $3$-manifolds with positive scalar curvature or non-negative scalar curvature and ``Schwarzschild asymptotics.'' It is thus natural to wonder whether an asymptotically flat manifold admits \emph{any} complete unbounded minimal surfaces whatsoever.  Theorem \ref{main} settles this question affirmatively, as long as the manifold does not contain any closed minimal surfaces. 

One reason to expect minimal surfaces to exist is the min-max theory of Almgren and Pitts \cite{Pitts}, which produces unstable minimal surfaces in general compact three-manifolds (even in those which do not contain any stable or area-minimizing surfaces). For closed manifolds of positive Ricci curvature, Marques--Neves have shown the existence of infinitely many minimal surfaces \cite{MarquesNeves:inf}.  Simon--Smith \cite{SimonSmith} used such methods to show that every closed Riemannian three-sphere contains a minimal embedded two-sphere (see also \cite{HaslhoferKetover}).  Similarly by sweeping out the manifold with planes, one might expect an asymptotically flat three-manifold to contain a minimal plane.  The difficulty is that an asymptotically flat three-manifold has infinite volume, and the slices of such a sweepout would also have infinite areas and thus the ``width" of such a family is not a sensible notion.

One can instead try to apply variational methods in a fixed (convex) ball $B_R(0)$ to obtain a minimal disk with boundary and then let $R\rightarrow\infty$.    The difficulty in carrying this out is that the sequence of minimal surfaces may run off to infinity as $R\rightarrow\infty$.  Indeed, in a non-flat asymptotically flat manifold $(M^{3},g)$ with non-negative scalar curvature, direct minimization is doomed to fail: by the work of the first-named author and Eichmair \cite{ChodoshEichmairCarlotto}, $(M^{3},g)$ cannot contain an unbounded area-minimizing surface. Thus, if one considers a large equatorial circle in $B_R(0)$ and let $\Sigma_R$ be a minimal disk solving the Plateau problem for this boundary curve, the limit of $\Sigma_R$ as $R\rightarrow\infty$ is guaranteed to be the empty set.  Similarly, index $1$ critical points obtained by min-max methods could potentially disappear in the limit.  

To emphasize the difficulty in controlling index $1$ surfaces obtained by min-max, one may consider a $3$-manifold $(M^{3},g)$ whose metric is asymptotic to the cone 
\[
\bar g_{\alpha} = dr^{2} + r^{2} \alpha^{2}g_{\bS^{2}},
\]
for $\alpha \in (0,1)$ (where $g_{\bS^{2}}$ is the standard round metric on the unit $2$-sphere). By \cite{ChodoshEichmairVolkmann} we know that $(M^{3},g)$ cannot contain \emph{any} unbounded immersed minimal surfaces of finite index. Hence, if one considers a sequence of index $1$ surfaces $\Sigma_{R}$ in $B_{R}(0)$ with respect to the metric $g$, the surfaces must necessarily run off to infinity as $R\to\infty$. Interestingly, the method developed in this paper also applies in this setting, showing that if $(M^{3},g)$ is asymptotic to $\overline g_{\alpha}$ and does not contain any closed minimal surfaces, then it contains properly embedded minimal planes through every point $p\in M$. These planes have quadratic area growth, but infinite index. We discuss the extension of Theorem \ref{main} to this setting in Section \ref{sec:index}. 

See also Section \ref{subsec:2d} below for a discussion of certain results overcoming the difficulty we have just described in the context of geodesics on surfaces. 

Finally, we note that even in the asymptotic region of $(M^{3},g)$ it is not clear that one can perturb a Euclidean minimal surface to a $g$-minimal surface; an obstruction to a particular such deformation was demonstrated in \cite{CarlottoMondino}. Moreover, such a perturbative technique has no hope of constructing surfaces through any fixed point $p\in M$ as we do in Theorem \ref{main}, since we do not assume that $g$ is close to the Euclidean metric in the compact part of the manifiold.

In this paper we overcome these difficulties by relying on degree theoretic techniques, rather than variational methods.  Degree theory was introduced in this context by Tomi--Tromba \cite{TT} and further developed by White \cite{White:spaceOfSurf,White:newAppDeg,White:spaceVaryingMetrics}.   Tomi--Tromba first applied it to show that a curve in the boundary of a convex body in $\mathbb{R}^3$ bounds an embedded minimal disk.   White extended the theory and proved (among other things) that a three-sphere with positive Ricci curvature contains an embedded minimal torus \cite{White:embeddedTorus}.  It was recently extended to the free boundary setting to prove that convex bodies contain embedded free boundary annuli \cite{MaximoNunesSmith}.

\begin{figure}[t]
\begin{tikzpicture}[scale=.6]
\begin{scope}
	\draw (0,0) circle (1.5); 
	\filldraw [opacity=.5] (0,0) circle (.1);
	\draw [very thick] plot [smooth] coordinates {(-1.5,0) (-.3,-1) (1.5,0)}; 
	\filldraw (-1.5,0) circle (.05);
	\filldraw (1.5,0) circle (.05);
\end{scope}
\begin{scope}[shift={(4,0)},rotate=45]
	\draw (0,0) circle (1.5); 
	\filldraw [opacity=.5] (0,0) circle (.1);
	\draw [very thick] plot [smooth] coordinates {(-1.5,0) (-.7,-.5) (.6,-.1) (1.5,0)}; 
	\filldraw (-1.5,0) circle (.05);
	\filldraw (1.5,0) circle (.05);
\end{scope}
\begin{scope}[shift={(8,0)},rotate=90]
	\draw (0,0) circle (1.5); 
	\filldraw [opacity=.5] (0,0) circle (.1);
	\draw [very thick] plot [smooth] coordinates {(-1.5,0) (-.3,-.1) (.3,.1) (1.5,0)};
	\filldraw (-1.5,0) circle (.05);
	\filldraw (1.5,0) circle (.05);
\end{scope}
\begin{scope}[shift={(12,0)},rotate=135]
	\draw (0,0) circle (1.5); 
	\filldraw [opacity=.5] (0,0) circle (.1);
	\draw [very thick] plot [smooth] coordinates {(-1.5,0) (.3,.4) (1.5,0)};
	\filldraw (-1.5,0) circle (.05);
	\filldraw (1.5,0) circle (.05);
\end{scope}
\begin{scope}[shift={(16,0)},rotate=180]
	\draw (0,0) circle (1.5); 
	\filldraw [opacity=.5] (0,0) circle (.1);
	\draw [very thick] plot [smooth] coordinates {(-1.5,0) (-.3,1) (.7,1) (1.5,0)};
	\filldraw (-1.5,0) circle (.05);
	\filldraw (1.5,0) circle (.05);
\end{scope}
\end{tikzpicture}
\caption{Under a $180^{\circ}$ flip, if a disk returns to the same side then it must pass through the origin at some point.}
\label{fig:flip}
\end{figure}
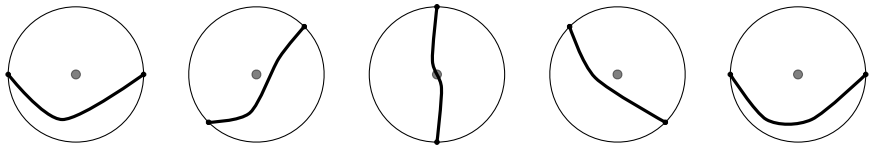
\begin{figure}[t]
\begin{tikzpicture}[scale=.6]
\begin{scope}
	\draw (0,0) circle (1.5); 
	\filldraw [opacity=.5] (0,0) circle (.1);
	\draw [blue, very thick] plot [smooth] coordinates {(-1.5,0) (-.4,-1) (1.5,0)}; 
	\draw [blue, very thick] plot [smooth] coordinates {(-1.5,0) (.5,-1.1) (1.5,0)}; 
	\draw [red, very thick] plot [smooth] coordinates {(-1.5,0) (0,.5) (1.5,0)}; 
	\draw [red, very thick] plot [smooth] coordinates {(-1.5,0) (.3,1) (1.5,0)}; 
	\filldraw (-1.5,0) circle (.05);
	\filldraw (1.5,0) circle (.05);
\end{scope}
\end{tikzpicture}
\caption{Considering the ``red'' and ``blue'' disks, we see that if no disk intersects the origin during the ``flip,'' the number of disks is even.}\label{fig:flip-even}
\end{figure}
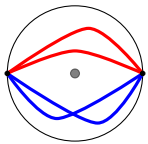
  Fix a large convex ball $B_R(0)$ in $M$.  We would like to produce a minimal disk passing through the origin (as then the limit as $R\rightarrow\infty$ would not be the empty set).  By the degree theory of White, it follows that an equatorial circle $C$ in the $xy$ plane in $\partial B_R(0)$ bounds an \emph{odd} number of embedded minimal disks.  However, assuming no minimal disks pass through the origin, we prove that any minimal disk bounded by $C$ in the southern hemisphere can be ``flipped" to another such disk in the northern hemisphere. See Figure \ref{fig:flip}.  Thus the number of minimal disks bounded by $C$ is \emph{even}. See Figure \ref{fig:flip-even}. This gives a contradiction and from it we obtain the existence of a minimal disk passing through the origin.  As the argument is indirect, we obtain no information about the Morse index of the minimal disk obtained.

Roughly speaking, the point is that minimal disks in the northern hemisphere pair off bijectively with those in the southern, and there must be some disk in the middle which ``flips" to itself in order to have an odd number of disks.  The rigorous argument and the precise notion of ``flipping" comes from the fact that $\mathbb{R}^3\setminus\{\mbox{origin}\}$ has two distinct isotopy classes of embedded two-spheres.

To apply degree theory in this setting and to take a limit as $R\rightarrow\infty$ we need area and curvature bounds for minimal disks with certain kinds of boundaries, which we also establish. A difficulty here is that we do not have any a priori control on the surfaces, since they are not constructed variationally. Instead, we will use curvature estimates based on the fact that the surfaces are disks. Schoen--Simon \cite{SchoenSimon} have proven that minimal disks $\Sigma$ in $\R^{3}$ with bounded area have curvature bounds away from $\partial\Sigma$. In a general Riemannian manifold, these curvature estimates might not apply (they would require that $\Sigma$ intersected any sufficiently small ball in a disk). Thus, we rely instead on the curvature estimates of White \cite{White:curvatureEstimates}. To apply these estimates, we must show that the surfaces have bounded area and controlled intersection with $\partial B_{r}(0)$ for $r$ sufficiently large.

In $\R^{3}$, we have the following isoperimetric inequality for minimal surfaces: if $\Sigma \subset B_{R}(0)\subset \R^{3}$ has $\partial\Sigma\subset \partial B_{R}(0)$, then taking $X = r\partial_{r}$ in the first variation, we find that
\[
2\area(\Sigma) = \int_{\Sigma} \Div_{\Sigma} X \, d\mu = \int_{\partial\Sigma} \bar g(\eta,X) d\mu \leq R \length(\partial\Sigma). 
\]
Such an estimate holds in the asymptotic region of an asymptotically flat manifold as well. However, notice that as $R\to\infty$, an estimate of this form will not give \emph{local} area bounds, since if $\length(\partial\Sigma) =O( R)$, then the estimate only implies $\area(\Sigma) \leq O(R^{2})$. In $\R^{3}$, this would be sufficient to prove local area bounds by the monotonicity formula, but here the error terms in the monotonicity formula might be too large for such an argument. 

Instead, we choose $\partial\Sigma$ to be very close to an equator in $\partial B_{R}(0)$ and use the above computation, along with a continuity argument to prove that $\Sigma$ intersects $\partial B_{r}(0)$ in a nearly equatorial circle, for all $r$ large. Carrying this out carefully will prove area bounds for $\Sigma$ outside of a fixed compact set. Finally, to prove area bounds in the fixed compact set, we rely on an isoperimetric inequality of White \cite{White:isoperimetric}, which requires that $M$ does not contain any minimal surfaces.

The assumption that $M$ contain no closed embedded minimal surfaces seems essential for the argument in its current form, since White has shown \cite[Theorem 4.1]{White:isoperimetric} that in the presence of closed embedded minimal surfaces, it is \emph{always} possible to find minimal disks bounded by well behaved curves, but with curvature and area blowing up. In fact, the logic of our construction in the proof of Theorem \ref{main} is somewhat analogous to White's construction of these misbehaving disks.

It is natural to wonder whether the minimal planes obtained by Theorem \ref{main} have index $0$ or $1$ in general (we show that as long as the metric satisfies a slightly stronger decay condition, the of the minimal plane index is finite in Proposition \ref{prop:index}, but do not estimate it explicitly). It also seems natural to conjecture that if $(M^{3},g)$ is asymptotically flat with $\partial M$ consisting of closed minimal surfaces then there is an unbounded minimal surface in $(M^{3},g)$ with (possibly empty) free boundary on $\partial M$. This is supported by the situation in the Schwarzschild manifold defined (for $m>0$) by
\[
g = \left( 1 + \frac{m}{2|x|}\right)^{4} \bar g
\]
on $M = \{|x| \geq m/2\}$, where any Euclidean coordinate plane through $\{0\}$ clearly yields such a surface. It would be interesting to compute the index of these free-boundary annuli in the exact Schwarzschild metric. This should be possible by an ODE analysis. 

More generally, are these annuli and the horizon the only embedded minimal surfaces in Schwarzschild? The corresponding problem for embedded closed constant mean curvature surfaces was recently solved by Brendle \cite{Brendle:IHES}: such surfaces must be centered coordinate spheres.

While many authors have studied min-max methods in the non-compact setting \cite{Montezuma,KetoverZhou,CollinHauswirthMazetRosenberg,ChambersLiokumovich},
to our knowledge Theorem \ref{main} is the first such construction in a manifold of \emph{infinite volume}.  

\subsection{Analogous results for geodesics on surfaces} \label{subsec:2d} One dimension lower, i.e., for geodesics on surfaces, Bangert proved \cite{Bangert:totallyConvex,Bangert:escaping} that every complete two dimensional plane contains a complete geodesic escaping to infinity. Moreover, Bonk and Lang use a flip argument in \cite[Proposition 6.1]{BonkLang} that has a similar flavor to our techniques described above. More recently, Carlotto and De Lellis proved \cite{CarlottoDeLellis} that an asymptotically conical surface with non-negative Gaussian curvature contains infinitely many properly embedded geodesics with Morse index at most one, resolving (in the setting of asymptotically conical surfaces) the issue described above about controlling the drifting of min-max critical points as the boundary is sent to infinity. 

We emphasize that the arguments in the papers \cite{Bangert:totallyConvex,Bangert:escaping,BonkLang,CarlottoDeLellis} make heavy use of the two-dimensional setting in various ways. In particular: (i) besides variational methods, geodesics can be constructed by solving an ODE initial value problem, (ii) the Gauss--Bonnet formula can be used in a strong way to control the behavior of geodesics on a surface, and (iii) geodesics have no extrinsic curvature and thus automatically satisfy curvature estimates. None of these three features carry over to the setting (minimal surfaces in three manifolds) we consider here.

\subsection{On the ``flip'' argument for Theorem \ref{main}} Let us give a more detailed sketch of the existence part of Theorem \ref{main}.
Let $B_R(0)$ be a large ball centered about the origin in $M$ and let $C_R:=\partial B_R(0)\cap\{z=0\}$ be an equatorial circle in $\partial B_R(0)$.   For $t\in [0,1]$ denote the equatorial circle
\begin{equation}
C_R^t:=\partial B_R(0)\cap \{z\cos(\pi t)=x \sin (\pi t)\}.
\end{equation}
The family $C_R^t$ consists of rotating the circle $C_R=C_R^0$ a full $180^\circ$ degrees in $\partial B_R(0)$ back to itself.  
For each $t\in [0,1]$ denote by $M_R^t$ the family of embedded minimal disks with boundary equal to $C_R^t$.  

 Our goal is to find a minimal embedded disk with boundary in $\partial B_R(0)$ passing through the origin.  We can assume toward a contradiction that \emph{none} of the disks in $\cup_t M_R^t$ pass through the origin.   

 Since large balls in $M$ are mean convex, one expects from work of Tomi--Tromba that there should be an \emph{odd} number of minimal disks in $B_R$ with boundary $C_R$.  However, assuming no disk in $\cup_t M_R^t$ passes through the origin we can show that the number of minimal disks with boundary in $C_R$ is \emph{even}.

To see this, note that $M_R^0$ consists of two types of disks depending on which ``side" of the disk the origin lies.  More precisely, the space of embedded disks in $\mathbb{R}^3\setminus\{\mbox{origin}\}$ with boundary $C_R$ has two connected components (both contractible).  
Let us thus denote the disks in $M_R^0$ as either ``red disks" $\text{Red}_R$ or ``blue disks" $\text{Blue}_R$, depending on the component in which they are contained.  
Assume that as $t$ changes, the family of disks $M_R^t$ changes continuously (this can be guaranteed after small perturbation of the curves $C_R^t$ by Smale's transversality theorem).   

Let $D_R^0$ be some disk in $\text{Blue}_R$.  As $t$ increases, the disk $D_R^0$ moves with its boundary to a disk $D_R^t$ in $M_R^t$ with boundary $C_R^t$, and finally at $t=1$ the disk returns back to a disk in $M_R^0$.  This gives a bijection
 \begin{equation}
\Phi: M_R^0\rightarrow M_R^0.
\end{equation}

We claim that
\begin{equation}\label{swap1}
\Phi(\text{Blue}_R)=\text{Red}_R.
\end{equation}
and similarly 
\begin{equation}\label{swap2}
\Phi(\text{Red}_R)=\text{Blue}_R.
\end{equation}

The reason for \eqref{swap1} and \eqref{swap2} is that if $\Phi$ mapped a point in $\text{Blue}_R$ to a point in $\text{Blue}_R$, then the family $\cup _t D_R^t$ would sweep-out out all of $B_R(0)$ and in particular $D_R^t$ would pass through the origin for some value of $t$.  Thus we would have found a minimal disk with boundary in $B_R(0)$ (although with rotated boundary from $C_R$), which contradicts our assumption that no such disk exists.  

But \eqref{swap1} and \eqref{swap2} imply that the cardinality of the set $\text{Blue}_R$ is the same as that of $\text{Red}_R$ and thus the number of minimal disks in $B_R^0$ is \emph{even}.  This is a contradiction. 

Thus for each $R$ we obtain a minimal disk $\Sigma_R$ in $B_R(0)$ passing through the origin with boundary in $\partial B_R(0)$ close to an equator.  In light of the curvature estimates we prove in this paper, we can take a limit of $\Sigma_R$ as $R\rightarrow\infty$ and obtain a smooth embedded minimal plane passing through the origin.  This completes the sketch of the ``flip'' argument used in Theorem \ref{main}.

\subsection{Organization of the paper} In Section \ref{sec:area} we prove the curvature bounds we need to take a limit of disks with boundaries in larger and larger balls.  In Section \ref{sec:degree} we introduce the degree theory of Tomi--Tromba as extended by White.  In Section \ref{sec:maintheorem} we prove Theorem \ref{main}. Section \ref{sec:index} includes a generalization of Theorem \ref{main} to the setting of asymptotically conical $3$-manifolds, as well as a discussion of the Morse index of the surfaces. 

\vspace{.4cm}

\noindent\emph{Acknowledgements:} O.C. was partially supported by the Oswald Veblen Fund and NSF grants DMS 1638352 and DMS 1811059. D.K. was partially supported by NSF Postdoctoral Fellowship DMS 1401996. We are grateful to the referee for a careful reading, and in particular for pointing out a mistake in the original version of Proposition \ref{prop:index}. O.C. would also like to thank Florian Johne for several useful comments on the first version of this article. 

\section{Area and Curvature Bounds}\label{sec:area}

In this section we prove curvature and area estimates for minimal disks in an asymptotically flat manifold with no closed embedded minimal surfaces. 

The following estimates are due to Anderson \cite{Anderson:curvature} and White \cite{White:curvatureEstimates}. We will use them to control the curvature of our surfaces in a large, but fixed ball. These estimates are somewhat related to those of Schoen--Simon \cite{SchoenSimon} but crucially do not require the surfaces to intersect small balls in topological disks.
\begin{theorem}[Curvature estimates]\label{thm:White-curv}
Let $N$ denote a compact Riemannian $3$-manifold with strictly mean convex boundary $\partial N$. Suppose that $\Sigma$ is an embedded minimal disk with $\partial\Sigma\subset \partial N$. Then, there is a constant $C$ depending on
\begin{itemize}
\item the Riemannian manifold $(N,g)$,
\item the area of $\Sigma$,
\item the $C^{2,\alpha}$-norm of $\partial\Sigma$ (i.e., the $C^{2,\alpha}$ norm of $\partial\Sigma$ as a map parametrized by arc length), and
\item the ``embeddedness'' of $\partial\Sigma$, i.e., $\max_{x\not =y\in\partial\Sigma} \frac{d_{\partial\Sigma}(x,y)}{d_{N}(x,y)}$,
\end{itemize}
so that the second fundamental form $A_{\Sigma}$ of $\Sigma$ satisfies $|A_{\Sigma}|\leq C$. 
\end{theorem}

Thus, to obtain curvature estimates, it will be crucial to obtain area bounds. We recall the following isoperimetric inequality due to White \cite{White:isoperimetric}.
\begin{theorem}[Isoperimetric inequality]\label{thm-iso}
Suppose that $N$ is a compact Riemannian $3$-manifold with strictly mean convex boundary $\partial N$. Assume that $N$ does not contain any closed embedded minimal surfaces. Suppose that $\Sigma$ is an embedded minimal surface with $\partial\Sigma\subset \partial N$. Then, there is a constant $C$ depending only on $(N,g)$ so that 
\[
\area_{g}(\Sigma) \leq C \length_{g}(\partial\Sigma).
\]
\end{theorem}

Suppose now that $(M^{3},g)$ is asymptotically flat and $\Sigma_{R}$ are embedded minimal disks in $B_{R}(0)$ whose boundary $\partial \Sigma_{R}$ is converging in $C^{2,\alpha}$ to an equator as $R\to\infty$ (after rescaling the picture to unit size).

Choose $\varepsilon>0$ sufficiently small so that any stationary integral $2$-varifold in $\R^{3}$ that is not a (multiplicity one) flat plane has density at infinity at least $1+2\varepsilon$ (this is possible by Allard's theorem; cf.\ \cite{White:Allard}). 

Fix $\sigma_{0}$ sufficiently large so that $|\nabla r| \leq 2$ and $D^{2}r^{2}\geq g$ for $r\geq \sigma_{0}$ (where $r$ is the usual Euclidean radial coordinate in the chart at infinity). Here, the gradient $\nabla$ and Hessian $D^{2}$ are taken with respect to the asymptotically flat metric $g$. The following quantity will be crucial for our proof of area and curvature estimates. 
\begin{definition}
Define $\sigma(R)$ to be the infimum of $\sigma \in [\sigma_{0},R]$ so that for all $\rho \in [\sigma,R)$, $\Sigma_{R}$ is transverse to $\partial B_{\rho}$, 
\[
\int_{\Sigma_{R} \cap \partial B_{\rho}(0)} \frac{1}{|\nabla_{\Sigma_{R}}r|} d\mu \leq 2\pi (1+\varepsilon) \rho,
\]
and after rescaling by $\rho^{-1}$, the curve $\Sigma_{R}\cap \partial B_{\rho}(0)$ is in the $\varepsilon$-neighborhood (in the $C^{2,\alpha}$ sense\footnote{We can use the $\Vert \cdot \Vert_{2,\alpha}^{*}$ norm from \cite[p.\ 244]{White:curvatureEstimates} to measure distance here.}) of the set of equatorial circles in $\partial B_{\rho}(0)$. 
\end{definition} 

Our goal will be to show that $\sigma(R)$ is uniformly bounded from above.

\begin{lemma}\label{goodboundary}
We have that $R^{-1}\sigma(R) \to 0$ as $R\to\infty$. 
\end{lemma}
\begin{proof}
Consider the vector field $X = r\nabla r$. Note that $D X = \frac 12 D^{2}r^{2} = g + o(1)$ as $r\to\infty$. Choose $\lambda_{R} \to 0$ so that $\lambda_{R} R \to \infty$ and so that $\Sigma_{R}$ intersects $\partial B_{\lambda_{R} R}(0)$ transversely (for example $\lambda_{R}\approx R^{-\frac 12}$ will suffice). Consider the vector field $X$ in the first variation formula for $\Sigma_{R}\setminus B_{\lambda_{R} R}(0)$:
\begin{align*}
& (2+o(1)) \area_{g}(\Sigma_{R}\setminus B_{\lambda_{R} R}(0))\\
& = \int_{\Sigma_{R} \setminus B_{\lambda_{R} R}(0)} \Div_{\Sigma} X \, d\mu\\
& = \int_{\partial \Sigma_{R}} g(\eta,X) \, d\mu - \int_{\Sigma_{R}\cap \partial B_{\lambda_{R} R}(0)} g(\eta,X) \, d\mu\\
& \leq (2\pi + o(1))R^{2},
\end{align*} 
where the $o(1)$ terms are as $R\to\infty$. Here, $\eta$ is the outwards pointing unit normal to $\partial(\Sigma_{R} \setminus B_{\lambda R})$ (so the second term on the third line was negative, and thus could be discarded). 

Consider $\tilde\Sigma_{R} : = R^{-1}(\Sigma_{R} \setminus B_{\lambda_{R} R})$, along with the associated rescaled metric $\tilde g$. Note that $\area_{\tilde g}(\tilde\Sigma_{R}) \leq \pi + o(1)$. Denote by $\tilde V$, a stationary integral varifold in $\R^{3}\setminus\{0\}$ of $\tilde\Sigma_{R}$ so that $\tilde\Sigma_{R}$ converges to $\tilde V$ in the varifold sense (after passing to a subsequence). It is clear that $\tilde V$ extends (cf.\ \cite[Lemma D.4]{CESY}) to a stationary integral varifold in $B_{1}(0)$ with $0 \in \supp \tilde V$ and $\Vert \tilde V\Vert (B_{1}(0)) \leq \pi$. Thus, the monotonicity formula implies that $\tilde V$ is the varifold associated to $\tilde\Sigma$, a flat disk through the origin with multiplicity one. 

Now, by White's version of Allard's interior and boundary regularity theorem \cite{White:Allard} we see that $\tilde \Sigma_{R}$ converges with multiplicity one in $C^{2,\alpha}$ on compact subsets of $\R^{3}\setminus\{0\}$ to $\tilde\Sigma$. Observe that $\tilde\Sigma$ intersects each $\partial B_{r}(0)$ transversely in an equatorial curve and 
\[
\int_{\tilde\Sigma\cap\partial B_{r}(0)}  \frac{1}{|\nabla_{\tilde\Sigma} r|}d\bar \mu = 2\pi r,
\]
for all $r \in (0,1]$. Thus, for any $r \in (0,1]$, we may take $R$ sufficiently large so that $R^{-1}\sigma(R) \leq r$. This proves the claim. 
\end{proof}

\begin{proposition}\label{goodboundary2}
The quantity $\sigma(R)$ is uniformly bounded from above as $R\to\infty$. 
\end{proposition}
\begin{proof}
The argument is somewhat similar to Lemma \ref{goodboundary}. However, we proceed here by contradiction. To this end, assume that $\sigma(R) \to \infty$ as $R\to\infty$. We will rescale $\Sigma_{R}$ by $\sigma(R)^{-1}$ to produce a contradiction.

By definition of $\sigma(R)$, we find that 
\[
\length_{g}( \Sigma_{R} \cap \partial B_{2\sigma(R)}(0)) \leq C\sigma(R)
\]
Hence, by considering $X=r\nabla r$ in the first variation as in Lemma \ref{goodboundary}, we see that 
\[
\area_{g}(\Sigma_{R} \cap (B_{2\sigma(R)}(0)\setminus B_{\sigma_{0}}(0)) \leq C \sigma(R)^{2}.
\]
Now, consider $\tilde\Sigma_{R} : = \sigma(R)^{-1}(\Sigma_{R}\setminus B_{\sigma_{0}})$. By Lemma \ref{goodboundary}, the boundary components of $\tilde\Sigma_{R}$ are eventually disjoint from any compact subset of $\R^{3}\setminus\{0\}$. 

By the definition of $\sigma(R)$ and the co-area formula, we have that for $\rho > 1$,
\begin{equation}\label{eq:tilde-sigma-quad-area}
\area_{\tilde g}(\tilde\Sigma_{R} \cap (B_{\rho}(0)\setminus B_{1}(0))) \leq \pi (1+\varepsilon) (\rho^{2}-1).
\end{equation}
Putting this together, $\tilde \Sigma_{R}$ has uniformly bounded area on compact subsets of $\R^{3}\setminus\{0\}$. Thus, we may pass to a subsequence and find a stationary integral varifold $\tilde V$ in $\R^{3}$ so that $\tilde\Sigma_{R}$ converges to $\tilde V$ away from $\{0\}$. Moreover, by \eqref{eq:tilde-sigma-quad-area}, $\tilde\Sigma$ has quadratic area growth, and density at infinity bounded above by $\pi (1+\varepsilon)$. Thus, a standard argument shows that the density at infinity is $\pi$, and $\tilde V$ is the varifold associated to a plane through the origin with multiplicity one. As before, Allard's theorem implies that the convergence of $\tilde\Sigma_{k}$ to $\tilde V$ occurs in $C^{2,\alpha}$ on compact subsets of $\R^{3} \setminus \{0\}$. 

Thus, we find that for $k$ large, $\tilde\Sigma_{k}$ is transverse to $\partial B_{r}(0)$ for all $r\approx 1$,
\[
\int_{\tilde\Sigma_{k}\cap \partial B_{r}(0)} \frac{1}{|\nabla_{\tilde\Sigma_{R}} r|}d\mu = 2\pi (1+o(1)) r
\]
as $k\to\infty$, and $\tilde\Sigma_{k}\cap \partial B_{r}(0)$ is converging in $C^{2,\alpha}$ to an equator in $\partial B_{r}(0)$ as $k\to\infty$. This contradicts the definition of $\sigma(R)$ after rescaling. 
\end{proof}

Now, taking $\sigma_{0}$ larger if necessary, using the isoperimetric inequality in Theorem \ref{thm-iso} as combined with the definition of $\sigma(R)$ we find that $\Sigma_{R}$ has uniformly bounded area inside of $B_{\sigma_{0}}(0)$ and uniform quadratic area growth outside of $B_{\sigma_{0}}(0)$. Furthermore, we have that $\Sigma_{R} \cap \partial B_{\rho}(0)$ is close in $C^{2,\alpha}$ to an equatorial circle for any $\rho\geq \sigma_{0}$. This allows us to apply the curvature estimates of Theorem \ref{thm:White-curv} to obtain the following compactness theorem:

\begin{proposition}[Compactness]\label{compactness}
Let $M$ denote an asymptotically flat manifold diffeomorphic to $\mathbb{R}^3$ which contains no closed embedded minimal surfaces. 

Let $\Sigma_{R}$, be a sequence of embedded minimal disks in $M$ containing $p\in M$ with $\partial\Sigma_{R}\subset\partial B_R(0)$ and $R\rightarrow\infty$.  Suppose in addition that after rescaling to unit size, $\partial\Sigma_{R}$ is converging to an equator in $\partial B_{R}(0)$ in the $C^{2,\alpha}$ topology. Then, a subsequence of $\Sigma_{R}$ converges smoothly on compact subsets of $M$ to a complete properly embedded minimal plane $\Sigma_\infty$ with $p\in\Sigma_{\infty}$. 

Furthermore, $\Sigma_{\infty}$ has quadratic area growth and for $\lambda\to\infty$, after passing to a subsequence $\lambda^{-1}\Sigma_{\infty}$ converges smoothly with multiplicity one on compact subsets of $\R^{3}\setminus \{0\}$ to a plane through the origin.
\end{proposition}

\begin{proof}
By Proposition \ref{goodboundary2}, the quantity $\sigma(R)$ is uniformly bounded as $R\to \infty$. Thus, we can take $\sigma_{0}$  even larger if necessary so that for all $R>\sigma_{0}$, and any $\rho \in [\sigma_{0},R)$, we have that $\Sigma_{R}$ is transverse to $\partial B_{\rho}(0)$,
\begin{equation}\label{eq:compactness-proof-coarea}
\int_{\Sigma_{R}\cap \partial B_{\rho}(0)} \frac{1}{|\nabla_{\Sigma_{R}}r|}d\mu \leq 2\pi(1+\varepsilon)\rho,
\end{equation}
and after rescaling by $\rho^{-1}$, the curve $\Sigma_{R}\cap \partial B_{\rho}(0)$ is in the $C^{2,\alpha}$ $\varepsilon$-neighborhood of an equator.

Thus, we see that $\Sigma_{R}\cap B_{\sigma_{0}}(0)$ has uniformly bounded area, by the isoperimetric inequality Theorem \ref{thm-iso}. This and the co-area formula (using \eqref{eq:compactness-proof-coarea}) show that there is $\Lambda > 0$ so that 
\begin{equation}\label{eq:compactness-proof-quad-growth}
\area_{g}(\Sigma_{R}\cap B_{\rho}(0)) \leq \Lambda  + \pi(1+\eps) (\rho^{2} - \sigma_{0}^{2})
\end{equation}
for all $\rho >\in [\sigma_{0},R]$ (where $\Lambda$ is independent of $R$). Moreover, for any $\rho \in[\sigma_{0},R]$, we have seen that $\Sigma_{R}\cap \partial B_{\rho}(0)$ is controlled in $C^{2,\alpha}$ (and uniformly ``embedded'' in the sense described in Theorem \ref{thm:White-curv}). Thus, by Theorem \ref{thm:White-curv} applied to $\Sigma_{R}\cap B_{\rho} \subset B_{\rho}(0)$, we have that the curvature of $\Sigma_{R}$ is uniformly bounded on compact subsets of $\R^{3}$. By the uniform quadratic area growth, so is the area, and thus we can pass to a subsequential (smooth) limit to find a properly embedded minimal plane $\Sigma_{\infty}$ with $p\in\Sigma_{\infty}$. 

The plane $\Sigma_{\infty}$ has quadratic area growth by \eqref{eq:compactness-proof-quad-growth}, so it remains to consider the blow-down limits $\lambda^{-1}\Sigma_{\infty}$. By the quadratic area growth, a subsequence converges in the varifold sense (on compact subsets of $\R^{3}\setminus\{0\}$) to a stationary integral varifold $V$ on $\R^{3}$ with the property that
\[
\Vert V\Vert(B_{\rho}(0)) \leq \pi(1+\eps).
\]
Note that here, exactly in the previous two proofs, we have used the standard extension property of stationary integral varifolds described in e.g.\ \cite[Lemma D.4]{CESY}. Now, by choice of $\eps$, Allard's theorem applies to show that $V$ is a multiplicity one plane in $\R^{3}$. Thus, the convergence happens smoothly with multiplicity one on compact subsets of $\R^{3}\setminus\{0\}$. 

Finally, we observe that $V$ must be a plane through the origin: since $p\in \Sigma_{\infty}$ and $\Sigma_{\infty}$ is connected, we can always find a point in $\Sigma_{\infty}\cap \partial B_{\eta\lambda}(0)$ for any $\eta>0$ fixed. Thus, $\lambda^{-1}\Sigma_{\infty}\cap \partial B_{\eta}(0) \not = \emptyset$. The monotonicity formula implies that there is a definite amount of area in a small ball around this point. This easily is seen to imply that the blow-down plane must pass through the origin. 
\end{proof}

\section{Degree Theory}\label{sec:degree}
In this section we introduce the degree theory of Tomi--Tromba \cite{TT} as later extended by White \cite{White:newAppDeg} needed for the proof of Theorem \ref{main}. 

Let $M$ denote a compact Riemannian three-ball with strictly mean convex boundary $\partial M$.   Let $D$ denote the flat unit disk in $\mathbb{R}^2$.  Let us call two maps $f_1,f_2:D\rightarrow M$ equivalent if $f_1=f_2\circ u$ for some diffeomorphism $u:D\rightarrow D$ fixing $\partial D$ pointwise.   Let $[f_1]$ denote the equivalence class of $f_1$. Let $$\mathcal{M}=\{[f]:f\in\mathcal{C}^{2,\alpha}(D,M)\mbox{ a minimal immersion with } f(\partial D)\subset \partial M\}.$$

We have the following theorem due to White \cite{White:spaceOfSurf} (generalizing earlier work of Tomi--Tromba in $\mathbb{R}^3$ \cite{TT}):
\begin{theorem}
The space $\mathcal{M}$ is a smooth Banach manifold and 
\begin{equation}
\Pi:\mathcal{M}\rightarrow\mathcal{C}^{2,\alpha}(\partial D, \partial M)
\end{equation}
given by 
\begin{equation}
\Pi([f])=f|_{\partial D}
\end{equation}
is a smooth Fredholm map of index $0$.
\end{theorem}

By Smale's infinite dimensional version \cite{Smale} of Sard's theorem it follows that the singular values of a Fredholm map are of the first category in the Baire sense, so in particular they contain no interior point. 
Since $\Pi$ is Fredholm of index $0$, for any regular value $y$ of the mapping $\Pi$, the set $\Pi^{-1}(y)$ is a $0$ dimensional manifold, and is locally the union of finitely many points \cite[Corollary 1.5]{Smale}.  

In order to assign a mod $2$ degree to the mapping $\Pi$ we need to restrict to subsets of $\mathcal{M}$ on which the mapping $\Pi$ is \emph{proper}.  Namely, we have the following:
\begin{theorem}[Mod $2$ Degree]
Let $\mathcal{M}'$ and $W$ be open subsets of $\mathcal{M}$ and $\mathcal{C}^{2,\alpha}(\partial D, \partial M)$ respectively, such that $W$ is connected and $\Pi:\mathcal{M}'\rightarrow W$ is proper.  Then for generic $\gamma\in W$, the number of elements 
$\Pi^{-1}(\gamma)\cap\mathcal{M}'$ is constant modulo $2$.
\end{theorem}

Finally, we have the following theorem \cite[Theorem. 2.1]{HoffmanWhite}:
\begin{theorem}
Suppose the Riemannian three-ball $M$ with mean convex boundary and containing no closed embedded minimal surfaces.  Let $\mathcal{M}'$ be the subset of $\mathcal{M}$ consisting of embeddings, and let $W:=\mathcal{C}^{2,\alpha}(\partial D, \partial M)$.  Then $\Pi$ restricted to $\mathcal{M}'$ is a proper map.  Moreover, the mod $2$ degree of $\Pi$ is equal to one.  In particular, a generic $\gamma\in\mathcal{C}^{2,\alpha}(\partial D, \partial M)$ bounds an odd number of embedded minimal disks.  
\end{theorem}

We also have the following which allows us to perturb curves in the Banach space $\mathcal{C}^{2,\alpha}(\partial D, \partial M)$ to be transverse to $\Pi$ \cite[Theorems 3.1 and 3.3]{Smale} (and thus have nice pre-images under $\Pi$):
\begin{theorem}[Smale's Transversality Theorem]\label{trans}
Let $\Gamma$ be a $C^{1}$ mapping $\Gamma:[0,1]\rightarrow\mathcal{C}^{2,\alpha}(\partial D, \partial M)$.  Then after arbitrarily small $C^{1}$ perturbation of $\Gamma$, one obtains a new mapping $\tilde{\Gamma}:[0,1]\rightarrow\mathcal{C}^{2,\alpha}(\partial D, \partial M)$
so that $\Pi^{-1}(\tilde{\Gamma}[0,1])\cap\mathcal{M}'$ is a smooth one-dimensional submanifold with boundary consisting of the finite set $\Pi^{-1}(\tilde{\Gamma}(0))\cap\mathcal{M}'$ and $\Pi^{-1}(\tilde{\Gamma}(1))\cap\mathcal{M}'$.

\end{theorem}
\section{Proof of Main theorem}\label{sec:maintheorem}
In this section we prove Theorem \ref{main}.  Thus let $M$ be an asymptotically flat three-manifold containing no closed embedded minimal surfaces.  Let $B_R(0)$ denote the Euclidean ball of radius $R$. Fix $p\in M$. We will always assume that $R$ is large enough so that $p\in B_{R}(0)$.

To apply degree theory, we need the following:
\begin{lemma}\label{mc}\label{isinside}
For $R$ large enough, the ball $B_R(0)$ is convex with respect to $g$. In particular, for $R$ large enough, any minimal disk with boundary in $\partial B_R(0)$ is contained entirely in $B_R(0)$.
\end{lemma}

Let $C_R:=\partial B_R(0)\cap\{z=0\}$ be the equatorial circle in $\partial B_R(0)$ in the $xy$-plane.   Let us consider a one parameter family of curves in $\partial B_R(0)$.  Namely for $t\in [0,2]$ denote the equatorial circle
\begin{equation}
C_R^t:=\partial B_R(0)\cap \{z\cos(\pi t)=x\sin(\pi t)\}.
\end{equation}
The family $C_R^t$ consists of rotating the circle $C_R=C_R^0$ a full $360^0$ degrees in $\partial B_R(0)$ back to itself.  

In fact we will be interested in only half of this family, namely the part with $t\in [0,1]$.  The path from $0$ to $1$ reverses the orientation from $C_R^0$ to $C_R^1$ and thus is not a closed loop in $\mathcal{C}^{2,\alpha}(\partial D, \partial M)$.

Since $M$ contains no embedded minimal surfaces and its boundary is mean convex (Lemma \ref{mc}) by Theorem \ref{trans} we can replace the curve $\{C_R^t\}_{t\in[0,2]}$ by a new curve $\{D_R^t\}_{t\in[0,2]}$ arbitrarily close to $\{C_R^t\}_{t\in[0,2]}$ so that:
\begin{proposition}[Degree is Odd]\label{odd}
The set $\mathcal{L} :=\Pi^{-1}(\cup_{t\in[0,1]}D_R^t)\cap\mathcal{M}'$ is a smooth one dimensional manifold.   Moreover, the closed curve $D_R^0$ in $\partial B_R(0)$ bounds an odd number of embedded minimal disks.  

We can arrange that the nearby curves $D_R^0$ and $D_R^1$ are in a connected regular neighborhood for $\Pi$ and thus bound the same number of embedded minimal disks. Finally, we can ensure that both curves $D_R^0$ and $D_R^1$ have images arbitrarily close to that of the equator in the $xy$-plane, $C_0$.
\end{proposition}

\begin{proof}
We can find a curve $\gamma$ arbitrarily close to $C^{0}_{R}$ which is a regular value of $\Pi$. Then, by concatenating small paths of curves on both ends of $\{C_{R}^{t}\}_{t\in[0,1]}$, we can obtain a path $\{\tilde D_{R}^{t}\}_{t\in[0,1]}$ so that $D^{0}_{R} = \gamma$, $D^{1}_{R} = -\gamma$ (i.e., $\gamma$ with the opposite orientation), and so that $D_{R}^{t}$ is arbitrarily close to $C_{R}^{t}$.

For $s\in[0,1]$ choose a path $E_{R}^{s}$ with $E_{R}^{0} = \gamma$ and so that for $s \in [\frac 12,1)$, $E_{R}^{s}$ is close to 
\[
\partial B_R(0)\cap \{z=Rs\}
\]
and a regular value of $\Pi$ for a.e.\ $s$ close to $1$. For $s$ close enough to $1$, it follows (see page 149 in \cite{White:newAppDeg}) that $E_R^s$ bounds precisely one embedded minimal disk.  Namely, $\Pi^{-1}(E_R^t)$ consists of one point.  Thus the mod $2$ degree of $\Pi$ on the set $\Pi^{-1}(\cup_{s\in[0,1]} E_{R}^{s})$ is odd.  Thus we see (cf.\ Theorem 2.1 in \cite{White:newAppDeg}) that $\gamma$ must bound an \emph{odd} number of disks.

Now, since $\gamma$ is a regular value for $\Pi$, any boundary curve which is sufficiently close to $\gamma$ will bound the same (odd) number of minimal surfaces as $\gamma$ (and they will be small perturbations of those bounded by $\gamma$). Then, by applying Theorem \ref{trans}, we can arrange for a small perturbation of $\{\tilde D_{R}^{t}\}_{t\in[0,1]}$ to $\{D_{R}^{t}\}_{t\in[0,1]}$ which is transverse to $\Pi$. If this perturbation is sufficiently small, the endpoints will still be in the regular neighborhood of $\gamma$, which is what we wanted. 
\end{proof}

 On the other hand, we have:
\begin{proposition}[Degree is Even]\label{even}
Suppose no disk in $\mathcal{L}$ passes through $p\in M$ fixed, then the number of disks bounded by $\gamma = D_R^0$ is even. 
\end{proposition}

See Figure \ref{fig:even} for an illustration of the proof of this proposition. 
\begin{proof}
Let us consider the smooth one-manifold $\mathcal{L}$.  The boundary of $\mathcal{L}$ consists of elements $\mathcal{A}:=\Pi^{-1}(D_R^0)$ together with elements in $\mathcal{B}:=\Pi^{-1}(D_R^1)$.  We want to compute the parity of the cardinality of $\mathcal{A}$ and $\mathcal{B}$.  By the classification of $1$-manifolds, each connected component of $\mathcal{L}$ with non-empty boundary has exactly two boundary points.  Some such components of $\mathcal{L}$ have both boundary points in either $\mathcal{A}$ or $\mathcal{B}$.  Let us denote these connected components by $\mathcal{L}''$ and $\mathcal{A}''$ ($\mathcal{B}''$) the subset of $\mathcal{A}$ (resp. $\mathcal{B}''$) joined by curves in $\mathcal{L}''$. 

Similarly let us denote by $\mathcal{A}'$ ($\mathcal{B}'$) the elements of  $\mathcal{A}$ (resp. $\mathcal{B}$) so that $\mathcal{L}$ connects each point in $\mathcal{A}'$ to one in $\mathcal{B}'$.  Let $\mathcal{L}'$ be the set of connected component of $\mathcal{L}$ with one boundary point in $\mathcal{A}$ and its other in $\mathcal{B}$.  Thus $\mathcal{L'}$ provides a bijection $\Phi$ between $\mathcal{A}'$ and $\mathcal{B}'$.   

Given a disk $D\in\mathcal{A}$, let $S_D$ denote the component of $\partial B_R(0)\setminus \partial D$ containing the South Pole of $\partial B_R(0)$, and $N_D$ the component containing the North Pole.  Let us say $D$ is a $\mbox{Blue}^0_R$ disk if the three-ball bounded by  $D\cup S_D$ does not contain $p$ in its interior (this definition is well-defined as no disk in $\mathcal{A}$ intersects $p$ by assumption and each disk in $\mathcal{A}$ is contained in $B_R(0)$ by Lemma \ref{isinside}).   If the three-ball bounded by $D\cup S_D$ does contain $p$, let us say $D$ is a ``red" disk $\mbox{Red}^0_R$.  Thus we partition $\mathcal{A}$ into $\mbox{Blue}^0_R$ and $\mbox{Red}^0_R$.  In the same way we partition $\mathcal{B}$ into $\mbox{Red}^1_R$ and $\mbox{Blue}^1_R$.

As in Proposition \ref{odd}, by choosing the perturbation $D_R^t$ of $C_R^t$ small enough, we obtain that the cardinality of $\mbox{Blue}^0_R$ is equal to that of $\mbox{Blue}^1_R$ and the cardinality of $\mbox{Red}^0_R$ is equal to that of $\mbox{Red}^1_R$.

We claim that in addition the cardinality of $\mbox{Blue}^0_R$ is equal to that of $\mbox{Red}^0_R$ modulo $2$.  Thus the cardinality of $\mathcal{A}$ is even and Proposition \ref{even} follows.  

Toward that end, we first consider the disks in $\mathcal{A}'\subset\mathcal{A}$ which are in bijective correspondence with $\mathcal{B}'$ by the map $\Phi$ described above.   We claim that 
\begin{equation}\label{swap}
\Phi(\mathcal{A}'\cap\mbox{Blue}^0_R)\subset  \mbox{Red}^1_R,
\end{equation}
and similarly 
\begin{equation}\label{swap3}
\Phi(\mathcal{A}'\cap\mbox{Red}^0_R)\subset  \mbox{Blue}^1_R.
\end{equation}

To prove \eqref{swap}, fix a disk $P\in \mathcal{A}'\cap\mbox{Blue}^0_R$.  Since $P$ is blue, it follows that $P\cup S_D$ is a two-sphere not containing the origin.  As we move $P$ along the curve in $\mathcal{L}$ linking it to a disk in $\mathcal{B}'$, we obtain a moving sequence of boundary curves $D_R^t$, starting at $D_R^0$ and ending at $D_R^1$, together with a moving sequence of disks $P_t$ with boundary $D_R^t$.   In this notation $P_0=P$ and $P_1=\Phi(P)$.  For each $t$ there is a choice of component $S_t$ of $\partial B_R(0)\setminus D_R^t$ so that at time $0$, $S_t$ is nearly all in the southern hemisphere, and the component $S_t$ varies continuously in $t$. For $t=1$, (after a $180^0$ rotation has been completed), $S_1$ is mostly in the northern hemisphere.  
Note that as no minimal disk in $\mathcal{L}$ hits $p$ and $S_t$ is contained in the boundary of the sphere of radius $R$, it follows that the two-sphere $P_t\cup S_t$ bounds a three-ball that is disjoint from the origin for all $t$.  Thus in particular $P_1\cup S_1$ bounds a ball disjoint from $p$.  It follows that $P_1$ is a red disk.  Thus $\Phi(P)$ is red as desired, establishing \eqref{swap} and \emph{mutatis mutandis}, \eqref{swap3}. 

Thus the cardinality of $\mathcal{A}'$ and thus also $\mathcal{B}'$ (as the set is in bijective correspondence with $\mathcal{A}'$) is even.  

It remains to consider the other elements of $\mathcal{A}$ which comprise the set $\mathcal{A}''$.  Arguing similarly to the above paragraph, one can see that a component of $\mathcal{L}''$ cannot join a blue disk in $\mathcal{A}''$ to a red disk in $\mathcal{A}''$.  Thus the only possibility is that each component of $\mathcal{L}''$ joints a red disk to a red disk, or a blue disk to a blue disk.  But anyway these contribute an even number of elements, and thus the cardinality of $\mathcal{A}''$ is even.  But since the cardinality of $\mathcal{A}$ is the sum of the cardinalities of $\mathcal{A}'$ and $\mathcal{A}''$, we obtain that this cardinality is even.  This completes the proof. 
\end{proof}

 \begin{figure}
 \begin{tikzpicture}[scale=.5]
\begin{scope}[shift={(0,2.5)}]
	\draw [->] (0,0) -- (0,9) node [above] {$\mathcal{L} = \Pi^{-1}(D_{R}^{t})$}; 
	\draw [->] (0,0) -- (17,0) node [right] {$t$}; 
	\filldraw [red] (0,8) circle (.1);
	\filldraw [red] (0,6) circle (.1);
	\filldraw [blue] (0,4) circle (.1);
	\filldraw [blue] (0,2) circle (.1);
	
	\draw [dashed, ->] (4,0) -- (4,9);
	\draw [dashed, ->] (8,0) -- (8,9);
	\draw [dashed, ->] (12,0) -- (12,9);
	\draw [dashed, ->] (16,0) -- (16,9);
	\filldraw [blue] (16,8) circle (.1);
	\filldraw [blue] (16,6) circle (.1);
	\filldraw [red] (16,4) circle (.1);
	\filldraw [red] (16,2) circle (.1);
	
	\draw plot [smooth, tension = 2] coordinates {(0,8) (9,8.3) (9,6.5)  (16,8)};
	\draw plot [smooth] coordinates {(0,6) (3,5) (11,6) (16,6)};
	\draw plot [smooth, tension = 2] coordinates {(0,4) (4,3) (0,2)};
	\draw plot [smooth, tension = 2] coordinates {(16,4) (12,3.5) (16,2)};
	
	\draw [decorate, decoration={brace,amplitude=10pt}] (-2,2) -- node [left = 9] {$\mathcal{A}$} (-2,8);
	\draw [decorate, decoration={brace,amplitude=5pt}] (-.5,6) -- node [left = 4] {$\mathcal{A}'$} (-.5,8);
	\draw [decorate, decoration={brace,amplitude=5pt}] (-.5,2) -- node [left = 4] {$\mathcal{A}''$} (-.5,4);
	
	\draw [decorate, decoration={brace,amplitude=10pt}] (18,8) -- node [right = 9] {$\mathcal{B}$} (18,2);
	\draw [decorate, decoration={brace,amplitude=5pt}] (16.5,8) -- node [right = 4] {$\mathcal{B}'$} (16.5,6);
	\draw [decorate, decoration={brace,amplitude=5pt}] (16.5,4) -- node [right = 4] {$\mathcal{B}''$} (16.5,2);
	
\end{scope}

\begin{scope}
	\draw (0,0) circle (1.5); 
	\filldraw [opacity=.5] (0,0) circle (.1);
	\draw [blue, very thick] plot [smooth] coordinates {(-1.5,0) (-.4,-1) (1.5,0)}; 
	\draw [blue, very thick] plot [smooth] coordinates {(-1.5,0) (.5,-1.1) (1.5,0)}; 
	\draw [red, very thick] plot [smooth] coordinates {(-1.5,0) (0,.5) (1.5,0)}; 
	\draw [red, very thick] plot [smooth] coordinates {(-1.5,0) (.3,1) (1.5,0)}; 
	\filldraw (-1.5,0) circle (.05);
	\filldraw (1.5,0) circle (.05);
	\node at (0,-2.2) {$t=0$};
\end{scope}

\begin{scope}[shift={(4,0)}]
\begin{scope}[rotate={-45}]
	\draw (0,0) circle (1.5); 
	\filldraw [opacity=.5] (0,0) circle (.1);
	\draw [ very thick] plot [smooth] coordinates {(-1.5,0) (0,-1) (1.5,0)}; 
	\draw [ very thick] plot [smooth] coordinates {(-1.5,0) (0,.6) (1.5,0)}; 
	\draw [very thick] plot [smooth] coordinates {(-1.5,0) (.5,1) (1.5,0)}; 
	\filldraw (-1.5,0) circle (.05);
	\filldraw (1.5,0) circle (.05);
\end{scope}
	\node at (0,-2.2) {$t=\frac 1 4$};
\end{scope}

\begin{scope}[shift={(8,0)}]
\begin{scope}[rotate={-90}]
	\draw (0,0) circle (1.5); 
	\filldraw [opacity=.5] (0,0) circle (.1);
	\draw [very thick] plot [smooth] coordinates {(-1.5,0) (.2,1.1) (1.5,0)}; 
	\draw [very thick] plot [smooth] coordinates {(-1.5,0) (0,.5) (1.5,0)}; 
	\draw [very thick] plot [smooth] coordinates {(-1.5,0) (-.5,1) (1.5,0)}; 
	
	\filldraw (-1.5,0) circle (.05);
	\filldraw (1.5,0) circle (.05);
\end{scope}
	\node at (0,-2.2) {$t=\frac 12$};
\end{scope}

\begin{scope}[shift={(12,0)}]
\begin{scope}[rotate={-135}]
	\draw (0,0) circle (1.5); 
	\filldraw [opacity=.5] (0,0) circle (.1);
	\draw [ very thick] plot [smooth] coordinates {(-1.5,0) (-.3,-1) (1.5,0)}; 
	\draw [ very thick] plot [smooth] coordinates {(-1.5,0) (0,1) (1.5,0)}; 
	\draw [very thick] plot [smooth] coordinates {(-1.5,0) (.5,.9) (1.5,0)}; 
	\filldraw (-1.5,0) circle (.05);
	\filldraw (1.5,0) circle (.05);
\end{scope}
	\node at (0,-2.2) {$t=\frac 34$};
\end{scope}

\begin{scope}[shift={(16,0)}]
	\draw (0,0) circle (1.5); 
	\filldraw [opacity=.5] (0,0) circle (.1);
	\draw [red, very thick] plot [smooth] coordinates {(-1.5,0) (-.4,-1) (1.5,0)}; 
	\draw [red, very thick] plot [smooth] coordinates {(-1.5,0) (.5,-1.1) (1.5,0)}; 
	\draw [blue, very thick] plot [smooth] coordinates {(-1.5,0) (0,.5) (1.5,0)}; 
	\draw [blue, very thick] plot [smooth] coordinates {(-1.5,0) (.3,1) (1.5,0)}; 
	\filldraw (-1.5,0) circle (.05);
	\filldraw (1.5,0) circle (.05);
	\node at (0,-2.2) {$t=1$};
\end{scope}
\end{tikzpicture}
\caption{The setup for the proof of Proposition \ref{prop:index}. The disks in $\mbox{Blue}^0_R$, $\mbox{Red}^0_R$, $\mbox{Blue}^1_R$, and $\mbox{Red}^1_R$ are colored blue and red respectively at $t=0$ and $t=1$. As we show in the proof, the components in the $1$-manifold $\mathcal{L}$ that go between $t=0$ and $t=1$ produces a bijection $\Phi$ between $\mathcal{A}'$ and $\mathcal{B}'$. }
\label{fig:even}
\end{figure}
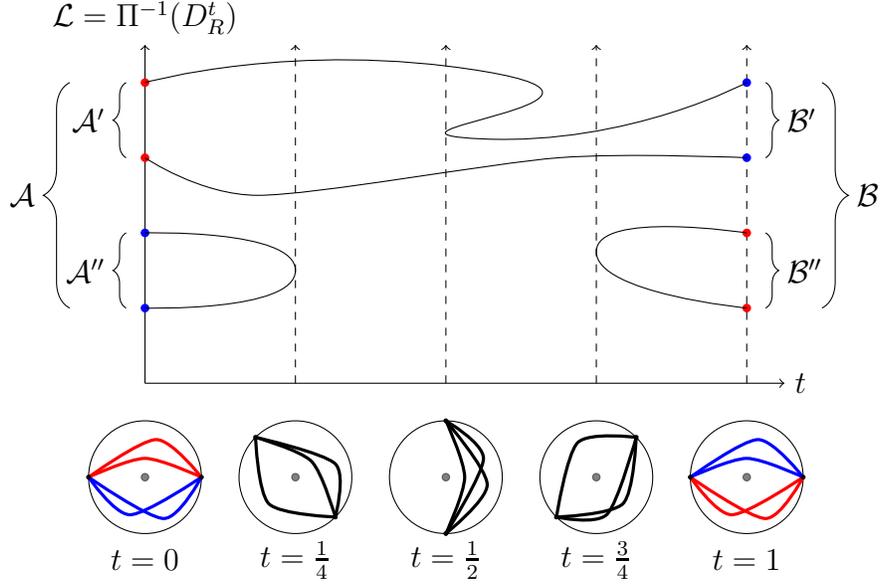

Since the conclusions of Propositions \ref{odd} and \ref{even} are in contradiction, it follows that the assumption of Proposition \ref{even} is false, and thus:
\begin{corollary}\label{gooddisk}
For $R$ large enough, $B_R(0)$ contains an embedded minimal disk passing through $p$ with boundary arbitrarily close to some equatorial circle $C^{t}_R$, where $t\in [0,1]$ depends on $R$.
\end{corollary}

We may thus combine this with the curvature and area estimates to complete the proof of the main result. 

\begin{proof}[Completion of proof of Theorem \ref{main}]

Let $R_i\rightarrow\infty$ be a sequence of radii, and let $\Sigma_i$ denote the embedded minimal disk with some boundary circle close to $C^{t_i}_{R_i}$ obtained from Corollary \ref{gooddisk}.  Denote by $\Sigma_\infty$ a subsequential limit of $\Sigma_i$ as $i\rightarrow\infty$ (using Proposition \ref{compactness}).  By Proposition \ref{compactness} $\Sigma_\infty$ contains $p$ and thus is non-empty. Moreover, the same proposition shows that $\Sigma_\infty$ is a smooth properly embedded minimal plane.  This completes the proof of Theorem \ref{main}.
\end{proof}

We remark that it should be possible to prove that $\Sigma_{\infty}$ has a unique tangent plane at infinity (cf.\ \cite[Lemma 14]{Carlotto}). It would be interesting to know if this tangent plane is the \emph{same} as the one containing the (limits of the) circles $C^{t_{i}}_{R}$. This would presumably imply that there is a full one-parameter family of minimal planes through any given point $p$. 

\section{Remarks related to the Morse index} \label{sec:index}

In this section we discuss the index of the minimal planes obtained in Theorem \ref{main}. We begin by discussing a related setting in which the disks $\Sigma_{R}$ have \emph{unbounded} index. We say that a metric on $M^{3}$ is asymptotically conical if $M$ is diffeomorphic to $\R^{3}$ and in the associated coordinates $g=\bar g_{\alpha} +b$ where 
\[
\bar g_{\alpha} = dr^{2} + r^{2}\alpha^{2}g_{\mathbb{S}^{2}}. 
\]
and $|b| + |x| |D_{\bar g_{\alpha}}  b| + |x|^{2}|D_{\bar g_{\alpha}}^{2}b| = o(1)$.  

\begin{theorem}\label{thm:AC}
Let $(M^{3},g)$ be an asymptotically conical $3$-manifold containing no closed embedded minimal surfaces. For every point $p\in M$ there exists a complete properly embedded minimal plane containing $p$. If the cone parameter satisfies $\alpha \in (0,1)$ each plane has infinite Morse index. 
\end{theorem}

\begin{proof}
The existence proof proceeds exactly as that of Theorem \ref{main}, after noting that the vector field $X = r\partial_{r}$ satisfies $D_{\bar g_{\alpha}} X =2 \bar g_{\alpha} $. Finally, the statement about the Morse index is a consequence of  \cite[Remark 10]{ChodoshEichmairVolkmann}.
\end{proof}

Now, for any \emph{asymptotically flat} $(M^{3},g)$ with no closed interior minimal surfaces, it is not hard to construct $g_{j}$ asymptotically conical (with $\alpha_{j}\to 1$) so that $g_{j}$ converges locally smoothly to $g$ and $(M^{3},g_{j})$ contains no closed embedded minimal surfaces. Through any $p\in M$, we can consider the sequence of minimal planes $\Sigma_{j}$   with respect to $g_{j}$ as constructed in Theorem \ref{thm:AC}. By appropriately modifying the arguments to prove compactness, we see that (after passing to a subsequence) $\Sigma_{j}$ converges locally smoothly to a minimal plane $\Sigma$ with respect to $g$ still containing $p$. One might expect $\Sigma$ still to have infinite Morse index. Surprisingly, this is not the case as long as we impose slightly stronger decay assumptions on the metric (as we show in the next proposition). Thus, the index of the $\Sigma_{j}$ ``drifts to infinity'' as the asymptotic cone angle parameter $\alpha$ tends to $1$. 

\begin{proposition}\label{prop:index}
Consider $(M^{3},g)$ asymptotically flat. Assume that the asymptotically flat metric $g$ satisfies the stronger decay condition:\footnote{Note that these conditions are still much weaker than is usually considered.} $g=\bar g+b$ where
\begin{equation}\label{eq:stronger-AF}
|b| + |x| |\bar D b| + |x|^{2} |\bar D^{2} b| = O(r^{-\tau})
\end{equation}
for some $\tau > 0$. Suppose that $\Sigma$ is an unbounded minimal surface in $(M^{3},g)$ so that
\begin{itemize}
\item $\Sigma$ has quadratic area growth and
\item for $\lambda\to\infty$, after passing to a subsequence, $\lambda^{-1}\Sigma$ converges in $C^{2,\alpha}_{\textrm{loc}}(\R^{3}\setminus\{0\})$ to a (multiplicity one) plane through the origin.
\end{itemize}
Then $\Sigma$ has finite Morse index. 
\end{proposition}

We note that the argument used here should extend (using e.g., arguments from \cite{ChodoshEichmairCarlotto}) to show \emph{equivalence} of finite index and finite total curvature for embedded minimal surfaces in asymptotically flat $3$-manifolds.\footnote{Depending on the hypothesis concerning $(M,g)$ (e.g., non-negative scalar curvature, Schwarzschild asymptotics, etc.) it may be necessary (or not) to assume quadratic area growth for the minimal surface.} We note that in the case of ambient $\R^{3}$ the equivalence of finite index and finite total curvautre is a well known result of Fischer-Colbrie \cite{FischerColbrie}.

\begin{proof}
We begin by observing that because $\lambda^{-1}\Sigma$ is close to a plane in $C^{2}_{\textrm{loc}}(\R^{3}\setminus\{0\})$ for $\lambda$ sufficiently large, we see that 
\[
\int_{\Sigma\cap \partial B_{R}} \kappa d\mu = 2\pi + o(1).
\]
as $R\to\infty$. Moreover, by convexity of large coordinate balls, it is clear that $\Sigma\cap B_{R}$ is a disk. Thus, Gauss--Bonnet yields
\[
\int_{\Sigma\cap B_{R}} K_{\Sigma} d\mu = o(1)
\]
as $R\to\infty$. On the other hand, the Gauss equations give
\[
2 K_{\Sigma} = R_{g} - 2\Ric_{g}(\nu,\nu) - |A_{\Sigma}|^{2} = O(r^{-2-\tau}) - |A_{\Sigma}|^{2},
\]
where we have used \eqref{eq:stronger-AF} to estimate the scalar curvature $R_{g}$ and Ricci curvature $\Ric_{g}$ of $g$. Because $\Sigma$ has quadratic area growth, a simple estimate on dyadic annuli gives
\[
\int_{\Sigma} O(r^{-2-\tau}) < \infty.
\]
Thus, 
\[
\int_{\Sigma} |A_{\Sigma}|^{2} < \infty. 
\]
This implies that
\begin{equation}\label{eq:improved-sff-decay}
|A_{\Sigma}| = O(r^{-1-\delta})
\end{equation}
for some $\delta>0$ by the work of Bernard--Riviere \cite[Corollary I.1]{BernardRiviere} (clearly $\Sigma$ is embedded outside of a compact set by blow-down assumption on $\lambda^{-1}\Sigma$).\footnote{We note also the work of Carlotto \cite{Carlotto} that proves such an estimate under the \emph{a priori} assumption that $\Sigma$ has is stable outside of a compact set (and under stronger asymptotic decay conditions of the metric). } 

We prove that for $\mu>0$ to be chosen, the function $\varphi(x) = 1-|x|^{-\mu}$ satisfies the following inequality outside of a compact set
\begin{equation}\label{eq:stab}
\Delta_{\Sigma} \varphi + (|A_{\Sigma}|^{2} + \Ric_{g}(\nu,\nu))\varphi \leq 0. 
\end{equation} 
Since $\varphi$ is positive, this implies that $\Sigma$ is stable outside of a compact set. This will then imply\footnote{In flat $\R^{3}$ there is a well known but indirect proof by Fischer-Colbrie \cite{FischerColbrie} that stability of a minimal surface outside of a compact set is equivalent to finite index. This proof does not seem to extend to the present situation; this is why we appeal to \cite{Devyver} here.} that $\Sigma$ has finite Morse index by work of Devyver \cite{Devyver}.

To establish \eqref{eq:stab}, note that $\Ric = O(r^{-2-\tau})$ by the assumed asymptotically flat condition \eqref{eq:stronger-AF}. On the other hand, a straightforward blow-down argument, using the fact that 
\[
\Delta_{\R^{2}} r^{-\mu} = \mu^{2} r^{-2-\mu}
\] 
shows that
\[
\Delta_{\Sigma} \varphi = - r^{-2-\mu}(\mu^{2}+o(1)) 
\]
as $|x|\to\infty$. Hence,
\[
\Delta_{\Sigma} \varphi + (|A_{\Sigma}|^{2} + \Ric_{g}(\nu,\nu))\varphi = -(\mu^{2} + o(1))r^{-2-\mu} + O(r^{-2-\tau}) + O(r^{-2-\delta})
\]
which is negative for $r$ sufficiently large, as long as we choose $0 < \mu < \min\{\tau,\delta\}$ (we recall that $\tau>0$ is the constant in \eqref{eq:stronger-AF} while $\delta>0$ is the constant in \eqref{eq:improved-sff-decay}). This completes the proof.
\end{proof}

\begin{remark} We give an example to show that Proposition \ref{prop:index} is false without the stronger notion of asymptotic flatness assumed there. Our construction follows a construction of Grigor'yan and Nadirashvili \cite[Section 2.6]{GN:schrod} modified in a straightforward manner to the present setting.

Consider a metric of the form
\[
g = dr^{2} + h(r)^{2}g_{\mathbb{S}^{2}}
\]
where $h(r)$ is smooth and satisfies $h(r) = r(1-(\log r)^{-2})$ for $r$ sufficiently large and $h(r) = r^{2}$ for $r$ sufficiently small. It is clear that $(\R^{3},g)$ is asymptotically flat in the sense of Theorem \ref{main} but not in the stronger sense considered in Proposition \ref{prop:index}. 

Consider $\Sigma$ a totally geodesic plane in $(\R^{3},g)$, i.e. for any equator $\gamma: \mathbb{S}^{1}\to\mathbb{S}^{2}$, set
\[
\Sigma : = \{ (r,\gamma(\theta)) : r\in[0,\infty), \theta\in\mathbb{S}^{1}\}.
\]
That $\Sigma$ is totally geodesic (and thus minimal) follows from the symmetry of $(\R^{3},g)$. We claim that $\Sigma$ has infinite Morse index. It is easy to compute (cf.\ \cite[(2)]{Brendle:IHES})
\[
\Ric(\nu,\nu) = -  \frac{h''(r)}{h(r)} + \frac{1-h'(r)^{2}}{h(r)^{2}} \geq (r \log r)^{-2}
\]
for $r$ sufficiently large. Consider the function 
\[
\varphi(r) = (\log r)^{\frac 12} \sin\left(\frac 12 \log\log r\right)
\]
for $r \in [2\pi k,2\pi(k+1)]$ (taking $\varphi$ identically $0$ otherwise). Note that 
\[
(r\varphi'(r))' = -\frac 12 r^{-2}(\log r)^{-2} \varphi(r)
\]
We consider $\varphi$ in the second variation form of area for $\Sigma$. We find, for $k$ sufficiently large:
\begin{align*}
Q(\varphi,\varphi) &  = \int_{\Sigma} ( |\nabla_{\Sigma} \varphi|^{2} - (|A_{\Sigma}|^{2} + \Ric(\nu,\nu))\varphi^{2}) d\mu \\
& \leq 2\pi \int_{2\pi k}^{2\pi(k+1)} (\varphi'(r)^{2} - (r\log r)^{-2} \varphi(r)^{2}) r (1-(\log r)^{-2}) dr\\
& \leq 2\pi \int_{2\pi k}^{2\pi(k+1)} \left(  r \varphi'(r)^{2} - \frac 3 4 r^{-1}(\log r)^{-2} \varphi(r)^{2}\right) dr\\
& = 2\pi \int_{2\pi k}^{2\pi(k+1)} (- (r \varphi'(r))' - \frac 3 4 r^{-1}(\log r)^{-2} \varphi(r))\varphi(r) dr\\
& = 2\pi \int_{2\pi k}^{2\pi(k+1)} \left( \frac 1 2  r^{-1}(\log r)^{-2} \varphi(r) - \frac 3 4 r^{-1}(\log r)^{-2} \varphi(r)\right)\varphi(r) dr\\
& = - \frac {\pi}{2} \int_{2\pi k}^{2\pi(k+1)}   r^{-1}(\log r)^{-2} \varphi(r)^{2} dr\\
& < 0.
\end{align*}
Because this holds for all $k$ sufficiently large, $\Sigma$ has infinite index. 
\end{remark}

\bibliography{bib} 
\bibliographystyle{amsplain}
\end{document}